\documentclass[11pt]{article}
\usepackage{amsmath}
\usepackage{amsthm}
\usepackage{amssymb}
\usepackage{amscd}
\usepackage{relsize}
\usepackage{enumerate}
\usepackage{color}
\usepackage{graphicx}
\usepackage{capt-of} 
\newtheorem{teo}{Theorem}
\newtheorem{lema}[teo]{Lemma}
\newcommand{\sinc}{\operatorname{sinc}}
\newcommand{\supp}{\operatorname{supp}}
\DeclareMathOperator*{\esup}{ess\,sup}
\DeclareMathOperator*{\einf}{ess\,inf}
\newcommand{\espan}{\operatorname{span}}

\definecolor{gris}{gray}{.93}
\newenvironment{cgris*}[2]{\vspace{0.3cm} \noindent
\fcolorbox{gris}{gris}{\hspace{0.04cm} \parbox{.97 \textwidth}
{\vspace{0.3cm}{#2} \vspace{0.3cm}} \hspace{0.04cm}}
\vspace{0.3cm}}  {}
\title{{\bf Sampling formulas involving differences in shift-invariant subspaces: a unified approach}}
\author{
{\bf Antonio~G. Garc\'{\i}a}\thanks{E-mail:\texttt{agarcia@math.uc3m.es}} \,\, 
{\bf and \, Mar\'{\i}a~J. Mu\~noz-Bouzo}\thanks{E-mail:\texttt{mjmunoz@mat.uned.es}}. \,\,}
\date{}
\begin{document}
\maketitle
\begin{itemize}
\item[*] Departamento de Matem\'aticas, Universidad Carlos III de Madrid,
 Avda. de la Universidad 30, 28911 Legan\'es-Madrid, Spain.
\item[\dag] Departamento de Matem\'aticas Fundamentales, U.N.E.D., Senda del Rey 9, 28040 Madrid, Spain.
\end{itemize}
%%%%%%%%%%%%%%%%%%%%
\begin{abstract}
Successive differences on a sequence of data help to discover some smoothness features of this data. This was one of the main reasons for rewriting the classical interpolation formula in terms of such data differences. The aim of this paper is to mimic them to a sequence of regular samples of a function in a shift-invariant subspace allowing its stable recovery. A suitable expression for the functions in the shift-invariant subspace by means of an isomorphism with the $L^2(0,1)$ space is the key to identify the simple pattern followed by the  dual Riesz bases involved in the derived formulas. The paper contains examples illustrating different non-exhaustive situations including also the two-dimensional case.
\end{abstract}
%%%%%%%%%%%%%%%%%%%%%%%%%%%%%%%%%%%%%%%%%%%%%%%%%%%%%%%%%%%%%%%%%%%%%%

\noindent{\bf AMS}: 42C15; 94A20; 97N50.
%%%%%%%%%%%%%%%%%%%%%
\section{Statement of the problem}
\label{section1}
%%%%%%%%%%%%%%%%%
Consider the shift-invariant space $V^2_\varphi=\big\{ \sum_{n\in \mathbb{Z}} a_n\, \varphi(t-n)\,:\, \{a_n\}\in \ell^2(\mathbb{Z})\big\}$ of $L^2(\mathbb{R})$ generated by $\varphi \in L^2(\mathbb{R})$. The goal of this work is to investigate the existence of  stable sampling formulas for any $f\in V^2_\varphi$ from differences (or averages) obtained from a  regular sequence of samples $\{f(a+n)\}_{n\in \mathbb{Z}}$. That is, the involved samples are $p$ subsequences taken from 
\begin{equation}
\label{diferencias}
\big\{f(a+pn)\big\}_{n\in \mathbb{Z}}\,; \quad \big\{\Delta^k_+ f(a+pn)\big\}_{n\in \mathbb{Z}}\,;\quad \mbox{and/or}\quad \big\{\Delta^k_- f(a+pn)\big\}_{n\in \mathbb{Z}}\,, 
\end{equation}
where $\Delta^k_+ f(a+pn):=\Delta^{k-1}_+ f(a+pn+1)-\Delta^{k-1}_+ f(a+pn)$ denotes the $k$-th  forward difference operator, $\Delta^k_- f(a+pn):=\Delta^{k-1}_- f(a+pn)-\Delta^{k-1}_- f(a+pn-1)$ denotes the $k$-th  backward difference operator, $1\le k\le p-1$, and $a$ is a fixed number in $[0,1)$; here $\Delta_\pm^0=I$. Successive differences on samples $\big\{f(a+n)\big\}_{n\in \mathbb{Z}}$  can help us to discover some smoothness features on the original function $f$.

It is a well-known that, under some appropriate hypotheses (see Section \ref{subsection2-2} below),  a {\em Shannon-type sampling formula} as
\begin{equation}
\label{sf1}
f(t)=\sum_{n\in \mathbb{Z}} f(a+n)\, S_a(t-n)\,, \quad t\in \mathbb{R}\,,
\end{equation} 
holds in $V^2_\varphi$ for some sampling function $S_a\in V^2_\varphi$ such that the sequence $\big\{S_a(t-n)\big\}_{n\in \mathbb{Z}}$ is a Riesz basis for $V^2_\varphi$. As a consequence, the above sampling formula is an interpolation formula: Indeed,  for each $n\in \mathbb{Z}$, the function $S_a$ satisfies the interpolation property $S_a(a+n)=\delta_{n}$. 

The sampling formula \eqref{sf1} was first proved by Walter in \cite{walter:92}; since then, sampling theory in shift-invariant subspaces has been largely studied; see, for instance, Refs. \cite{aldroubi:01,aldroubi:05,chen:02,garcia:06,kang:11,kwon:08,sun:02,sun:99}. For the general theory of shift-invariant subspaces and their applications, see, for instance, Refs. \cite{aldroubi:94,aldroubi:92,boor1:94,boor2:94,lei:97,unser:00}.

\medskip

Roughly speaking, the aim of the paper is to rewrite formula \eqref{sf1} by replacing some samples by any other information given by differences, averages, etc.  Let us consider an easy example. Initially, we  write formula \eqref{sf1} as
\[
f(t)=\sum_{n\in \mathbb{Z}}\big\{ f(a+2n)\, S_a(t-2n)+f(a+2n+1)\, S_a(t-2n-1)\big\}\,, \quad t\in \mathbb{R}\,.
\]
Then, adding and subtracting $f(a+2n)\, S_a(t-2n-1)$ in each summand yields
\[
f(t)=\sum_{n\in \mathbb{Z}} \big\{f(a+2n)\,\big[S_a(t-2n)+S_a(t-2n-1)\big]+\Delta_+ f(a+2n)\,S_a(t-2n-1)\big\}\,, \quad t\in \mathbb{R}\,.
\]
The stable recovery of any $f\in V^2_\varphi$ from the samples $\big\{f(a+2n)\big\}_{n\in \mathbb{Z}}\cup \big\{\Delta_+ f(a+2n)\big\}_{n\in \mathbb{Z}}$ follows from the fact that the sequence 
\[
\big\{S_a(t-2n)+S_a(t-2n-1)\big\}_{n\in \mathbb{Z}}\cup \big\{S_a(t-2n-1) \big\}_{n\in \mathbb{Z}}
\]
turns out to be a Riesz basis for $V^2_\varphi$, or equivalently, there exist constants $0<A\le B$ such that 
\[
A\|f\|^2 \le \sum_{n\in \mathbb{Z}}\big\{|f(a+2n)|^2+|\Delta_+ f(a+2n)|^2 \big\} \le B\|f\|^2\quad \text{for each $f\in V^2_\varphi$}\,.
\]

\medskip

In this paper we derive, in a systematic way, stable sampling formulas involving forward and/or backward  differences of arbitrary order or averages of the sequence of samples $\{f(a+n)\}_{n\in \mathbb{Z}}$. The used technique will allow to deal also with central differences and averages; this is done in Section \ref{section3}. Thus we obtain a sort of counterpart to the family of finite interpolation formulas in the infinite setting of sampling in shift-invariant subspaces of $L^2(\mathbb{R})$.
The existence of an isomorphism $\mathcal{T}_\varphi$ between $L^2(0,1)$ and $V^2_\varphi$ and an associated representation of the functions in $V^2_\varphi$ (see \eqref{expression1} below) make the mathematical technique very simple: we identify the available samples as inner products with respect to a suitable Riesz basis in $L^2(0,1)$ for which we obtain its corresponding dual Riesz basis; then we derive the corresponding sampling formula in $V^2_\varphi$ by using $\mathcal{T}_\varphi$. Finally, the two-dimensional case is also treated in Section \ref{section4}.

%%%%%%%%%%%%%%%%%%%%%%%%%%%%%%%%%%%%
\section{Sampling theorem in a shift-invariant subspaces of $L^2(\mathbb{R})$}
\label{section2}
%%%%%%%%%%%%%%%%%%%%%%%%%%%%%%%%%%%%
For the sake of completeness we include the  preliminaries on sampling in shift-invariant subspaces needed in the sequel.
%%%%%%%%%%%%%%%%%%%%%%%%
\subsection{Some preliminaries on shift-invariant subspaces}
\label{subsection2-1}
%%%%%%%%%%%%%%%%%%%%%%%%
Assume that for $\varphi \in L^2(\mathbb{R})$ the sequence $\{\varphi(t-n)\}_{n\in \mathbb{Z}}$ is a {\em Riesz sequence} for $L^2(\mathbb{R})$, i.e., a Riesz basis for its closed span. A {\em Riesz basis} for a (separable) Hilbert space is the image of an orthonormal basis by means of an invertible bounded operator.
Thus, the (principal) shift-invariant subspace $V^2_\varphi:=\overline{\espan}\{\varphi(t-n)\}_{n\in \mathbb{Z}}$ can be described as 
\[
V^2_\varphi=\big\{ \sum_{n\in \mathbb{Z}} a_n\, \varphi(t-n)\,:\, \{a_n\}\in \ell^2(\mathbb{Z})\big\}\,.
\] 
We also assume that the generator $\varphi$ is continuous on $\mathbb{R}$ and the $1$-periodic function $\sum_{n\in \mathbb{Z}} |\varphi(t-n)|^2$ is bounded on $[0,1]$; this is equivalently to assume that $V_\varphi^2   $ is a shift-invariant subspace of continuous functions on $\mathbb{R}$ \cite{sun:99}. As a consequence, the functions in $V^2_\varphi$ are continuous on $\mathbb{R}$ and any $f\in V^2_\varphi$ is described on $\mathbb{R}$ as the pointwise sum $f(t)=\sum_{n\in \mathbb{Z}} a_n\, \varphi(t-n)$, \,$t\in \mathbb{R}$.

On the other hand, the shift-invariant space $V^2_\varphi$ is the image of $L^2(0,1)$ by means of the isomorphism: 
\[
    \begin{array}{rccl}
	\mathcal{T}_\varphi:& L^2(0,1) &\longrightarrow & V_\varphi^2\\
	  & \{{\rm e}^{-2\pi i nx}\}_{n\in \mathbb{Z}} &\longmapsto & \{\varphi(t-n)\}_{n\in \mathbb{Z}}\,,
    \end{array}
\]
which maps the orthonormal basis  $\{{\rm e}^{-2\pi inw}\}_{n\in \mathbb{Z}}$ for $L^2(0,1)$ onto
the Riesz basis $\{\varphi(t-n)\}_{n\in\mathbb{Z}}$ for $V_\varphi^2$. This isomorphism $\mathcal{T}_\varphi$ satisfies, for each $F\in L^2(0,1)$ and $n\in \mathbb{Z}$, the following {\em shifting property}
\begin{equation}
\label{shifting}
\mathcal{T}_\varphi \big[F(x)\,e^{-2\pi i n x}\big](t)=\mathcal{T}_\varphi[F](t-n)\,, \quad t\in \mathbb{R}\,.
\end{equation}
After some straightforward calculations (see \cite{garcia:05,garcia:14} for the details), any $f\in V^2_\varphi$ can be expressed as
\begin{equation}
\label{expression1}
f(t)=\big\langle F, K_t \big\rangle_{L^2(0,1)}\,, \quad t\in \mathbb{R}\,,
\end{equation}
where $f=\mathcal{T}_\varphi F$ and $K_t(x)=\sum_{n\in \mathbb{Z}} \overline{\varphi(t-n)}\, e^{-2\pi i nx} \in L^2(0,1)$. Notice that $K_t(x)=\overline{Z\varphi(t,x)}$, where 
$Z\varphi(t,x):=\sum_{n\in \mathbb{Z}}\varphi(t+n){\rm e}^{-2\pi i nx}$ is just the {\em Zak transform} of the function $\varphi$ (see \cite{ole:03} for properties and uses of the Zak transform). For $t\in \mathbb{R}$ and $m\in \mathbb{Z}$, the following property holds: $K_{t+m}(x)=e^{-2\pi i mx}K_t(x)$.

From expression \eqref{expression1} we easily deduce that evaluation functionals are bounded on $V^2_\varphi$, i.e.,  it is a {\em reproducing kernel Hilbert space} (RKHS). Thus, convergence in the $L^2(\mathbb{R})$-norm implies pointwise convergence which is uniform on $\mathbb{R}$ since the function $t \mapsto \|K_t\|^2=\sum_{n\in \mathbb{Z}} |\varphi(t-n)|^2$ is assumed to be bounded on $\mathbb{R}$. As a consequence, from now on we are only concerned with $L^2(\mathbb{R})$-norm convergence. See more details in \cite{garcia:14,garcia:15}.

%%%%%%%%%%%%%%%%%%%%%%%%
\subsection{Shannon-type sampling formula in $V^2_\varphi$}
\label{subsection2-2}
%%%%%%%%%%%%%%%%%%%%%%%%
Having in mind expression \eqref{expression1}, for $a\in [0,1)$ fixed and $n\in \mathbb{Z}$ we have
\[
f(a+n)=\langle F, K_{a+n}\rangle_{L^2(0,1)} = \langle F, {\rm e}^{-2\pi i n x}K_a \rangle_{L^2(0,1)}\,, \quad  F=\mathcal{T}_\varphi^{-1} f\,.
\]
Thus, the sequence of samples $\{f(a+n)\}_{n\in \mathbb{Z}}$ forms a {\em stable sampling set} for any $f\in V^2_\varphi$ if and only if the sequence 
$\big\{{\rm e}^{-2\pi i n x}K_a(x)\big\}_{n\in \mathbb{Z}}$ is a Riesz basis for $L^2(0,1)$. This occurs if and only if the inequalities $0<\|K_a\|_0\le \|K_a\|_\infty < \infty$ hold, where $\|K_a\|_0:=\einf_{x\in(0,1)}|K_a(x)|$ and  $\|K_a\|_\infty:=\esup_{x\in(0,1)}|K_a(x)|$.  Moreover, its dual Riesz basis is  $\big\{{\rm e}^{-2\pi i n x}/\overline{K_a(x)}\big\}_{n\in \mathbb{Z}}$ (see \cite{garcia:05,garcia:14} for the details). Therefore, for each $F\in L^2(0,1)$ we have the expansion
\[
F=\sum_{n\in \mathbb{Z}} \langle F, {\rm e}^{-2\pi i n x}K_a \rangle \, \frac{{\rm e}^{-2\pi i n x}}{\overline{K_a(x)}}=\sum_{n\in \mathbb{Z}} f(a+n)\, \frac{{\rm e}^{-2\pi i n x}}{\overline{K_a(x)}}\quad \text{ in $L^2(0,1)$}\,.
\]
Applying the isomorphism $\mathcal{T}_\varphi$ and its shifting property we obtain in $V^2_\varphi$ the Shannon-type sampling formula
\[
f(t)=\sum_{n\in \mathbb{Z}} f(a+n)\,S_a(t-n)\,, \quad t\in \mathbb{R}\,,
\]
where $S_a:=\mathcal{T}_\varphi\big(1/\overline{K_a}\big)\in V_\varphi^2$. The  convergence of the series is absolute due to the unconditional convergence of a Riesz basis expansion and it is also uniform on $\mathbb{R}$ since the shift-invariant space $V_\varphi^2$ is a RKHS.
%%%%%%%%%%%%%%%%%%%%%%%%%%
\subsubsection*{Some well-known examples}
%%%%%%%%%%%%%%%%%
A famous example of shift-invariant subspace in  $L^2(\mathbb{R})$ is the Paley-Wiener space of band-limited functions to an interval, say $[-\pi, \pi]$, i.e., 
\[
PW_\pi:=\big\{f\in L^2(\mathbb{R})\,:\, \supp \widehat{f}\subseteq [-\pi,\pi] \big\}\,.
\] 
It coincides with the shift-invariant subspace $V^2_{\text{sinc}}$ whose generator $\varphi$ is the {\em sinc function}, i.e.,  $\sinc t:=\sin \pi t/\pi t$. In $PW_\pi$, for any $a\in [0,1)$, the sampling formula
\[
f(t)=\sum_{n\in \mathbb{Z}} f(n+a)\, \frac{\sin \pi(t-n-a)}{\pi(t-n-a)}\,, \quad t\in \mathbb{R}\,,
\]
holds. In particular, taking $a=0$ we recover the famous Shannon's sampling formula. 

\medskip

Although Shannon's sampling theory has had an enormous impact, it has a number of problems as pointed out in \cite{unser:00}. 
From a practical point of view, the most important examples of shift-invariant spaces $V_\varphi^2$ are those generated by {\em $B$-splines}. Consider 
$\varphi :=N_m$ where $N_m$ is the $B$-spline of order $m-1$, i.e., $N_m:=N_1*N_1*\cdots*N_1$  ($m$ times) 
where $N_1:=\chi_{[0,1]}$ denotes the characteristic function of the interval $[0,1]$. It is known that the sequence $\big\{N_m(t-n)\big\}_{n\in \mathbb{Z}}$ is a Riesz basis for $V_{N_m}^2$ (see, for instance, \cite{ole:03}). For the quadratic and cubic splines the following sampling formulas hold:

\medskip

\noindent (1) {\em Quadratic Spline $N_3$}.  Here, taking $a=1/2$, for any $f \in V_{N_3}^2$ we obtain the sampling formula
\[
f(t)=\sum_{n\in \mathbb{Z}} f(n+\frac{1}{2})\, S_{1/2}(t-n)\,,\quad t\in \mathbb{R}\,,
\]
where  $S_{1/2}(t)=\sqrt{2}\sum_{n\in \mathbb{Z}} (2\sqrt{2}-3)^{|n+1|}\  N_3(t-n)$ (see Fig.\ref{figure1}). For $a=0$ the sampling result fails since $\|K_0\|_0=0$; see \cite{garcia:14} for the details.

\medskip

\noindent (2) {\em Cubic Spline $N_4$}. Here, taking $a=0$, for any $f\in V_{N_4}^2$ we obtain the sampling formula
\[
f(t)=\sum_{n\in \mathbb{Z}} f(n)\, S_0(t-n)\,, \quad t\in \mathbb{R}\,,
\]
where $S_0(t)=\sqrt{3}\sum_{n\in \mathbb{Z}} (-1)^n(2-\sqrt{3})^{|n|}\ N_4(t-n+2)$ (see Fig.\ref{figure1}). See \cite{garcia:14} for the details. In Ref. \cite{schoenberg:93} one can find more information on cardinal spline interpolation.

\vspace{0.5cm}
\begin{figure}
\centering
\includegraphics[width=6cm]{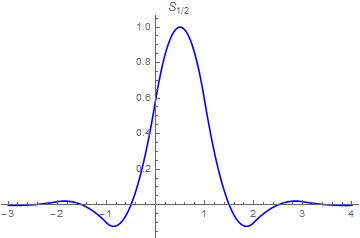}\qquad \includegraphics[width=6cm]{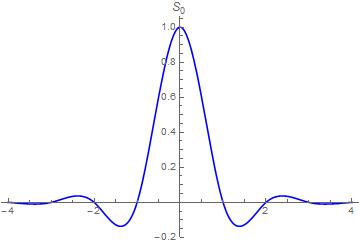}
\caption{Quadratic $S_{1/2}$ and cubic $S_0$ sampling funtions}
\label{figure1}
\end{figure}

%%%%%%%%%%%%%%%%%%%%%%%%
\subsection{The involved Riesz bases in $L^2(0,1)$ and their duals}
\label{subsection2-3}
%%%%%%%%%%%%%%%%%%%%%%%%
Having in mind expression \eqref{expression1}, for the {\em forward difference operator $\Delta_+$}  of $f\in V_\varphi^2$, i.e., $\Delta_+ f(t):=f(t+1)-f(t)$, $t\in \mathbb{R}$, we have
\[
\Delta_+ f(t)=\big\langle F, K_{t+1}-K_t \big\rangle_{L^2(0,1)}= \big\langle F, (e^{-2\pi i x}-1)\,K_t \big\rangle_{L^2(0,1)}\,, \quad t\in \mathbb{R}\,.
\]
In general, for the $k$-th forward difference operator we have the expression
\begin{equation}
\label{d+}
\Delta^k_+ f(t)=\big\langle F, (e^{-2\pi i x}-1)^k\,K_t \big\rangle_{L^2(0,1)}\,, \quad t\in \mathbb{R}\,.
\end{equation}
For the {\em backward difference operator $\Delta_-$} of $f\in V_\varphi^2$, i.e., $\Delta_- f(t):=f(t)-f(t-1)$, $t\in \mathbb{R}$, we have
\[
\Delta_- f(t)=\big\langle F, K_{t}-K_{t-1} \big\rangle_{L^2(0,1)}= \big\langle F, (1-e^{2\pi i x})\,K_t \big\rangle_{L^2(0,1)}\,, \quad t\in \mathbb{R}\,.
\]
In general, for the $k$-th backward difference operator we have the expression
\begin{equation}
\label{d-}
\Delta_-^k f(t)=\big\langle F, (1-e^{2\pi i x})^k\,K_t \big\rangle_{L^2(0,1)}\,, \quad t\in \mathbb{R}\,.
\end{equation}
Similar expressions can be also obtained from iterates of the {\em central difference operator} $\Delta_0$, i.e., $ \Delta_0 f(t):=f(t+1)-f(t-1)$, or from the {\em average operators} $\mu_+$, $\mu_-$ or $\mu_0$ defined as:
\[
\mu_+ f(t):=\frac{1}{2}[f(t+1)+f(t)]\,, \,\, \mu_- f(t):=\frac{1}{2}[f(t)+f(t-1)]\,, \,\,\mu_0 f(t):=\frac{1}{2}[f(t+1)+f(t-1)]\,.
\]
All these cases can be englobed by considering {\em generalized finite differences} defined as $\Delta_{m,n}^\mathbf{a} f(t):=\sum_{k=M}^N a_kf(t+k)$, where $\mathbf{a}=(a_k)\subset \mathbb{C}$ and $M,N\in \mathbb{Z}$.

\medskip

Going back to our previous example in Section \ref{section1}, we have for the samples $f(a+2n)$ and $\Delta_+ f(a+2n)$, $n\in \mathbb{Z}$:
\[
\begin{split}
f(a+2n)&=\big\langle F, e^{-2\pi i (2n) x} K_a \big\rangle_{L^2(0,1)}\,, \\ 
\Delta_+ f(a+2n)&=\big\langle F, e^{-2\pi i (2n+1) x} K_a-e^{-2\pi i (2n) x} K_a\big\rangle_{L^2(0,1)}\,.
\end{split}
\]
In other words, denoting $x_n:=e^{-2\pi i n x} K_a(x)$, $n\in \mathbb{Z}$, the questions are whether the sequence $\{x_{2n}\}_{n\in \mathbb{Z}} \cup \{x_{2n+1}-x_{2n}\}_{n\in \mathbb{Z}}$ is a Riesz basis for $L^2(0,1)$ and how to obtain, in the affirmative case, its dual basis. As it was mentioned in the introduction section, this is true: Indeed, denoting 
$y_n:=e^{-2\pi i n x}/\overline{K_a(x)}$, $n\in \mathbb{Z}$, the dual Riesz basis of $\{x_n\}_{n\in \mathbb{Z}}$, the dual Riesz basis of $\{x_{2n}\}_{n\in \mathbb{Z}} \cup \{x_{2n+1}-x_{2n}\}_{n\in \mathbb{Z}}$ is $\{y_{2n}+y_{2n+1}\}_{n\in \mathbb{Z}} \cup \{y_{2n+1}\}_{n\in \mathbb{Z}}$.

\medskip 

A general answer which involves the aforementioned generalized  finite differences is given in the following result:
\begin{lema}
\label{partition}
Let $\{x_n\}_{n\in \mathbb{Z}}$ be a Riesz basis for a separable Hilbert space $\mathcal{H}$ with dual Riesz basis $\{y_n\}_{n\in \mathbb{Z}}$. For a fixed $p\in \mathbb{N}$, consider a partition 
$\{x_{1n}\}_{n\in \mathbb{Z}}\cup \{x_{2n}\}_{n\in \mathbb{Z}} \cup \dots \cup \{x_{pn}\}_{n\in \mathbb{Z}}$ of $\{x_n\}_{n\in \mathbb{Z}}$ and let $M=\big( a_{jk}\big) \in \mathbb{C}^{p\times p}$ be a matrix of scalars. Then, the new sequence defined by using the rows of $M$ as
\begin{equation}
\label{lcriesz}
\big\{\sum_{k=1}^pa_{1k}x_{kn}\big\}_{n\in \mathbb{Z}}\cup \big\{\sum_{k=1}^pa_{2k}x_{kn}\big\}_{n\in \mathbb{Z}} \cup \dots 
\cup \big\{\sum_{k=1}^p a_{pk}x_{kn}\big\}_{n\in \mathbb{Z}}
\end{equation}
is a Riesz basis for $\mathcal{H}$ if and only if $\det M \neq 0$. In this case, its dual Riesz basis is given by 
\[
\big\{\sum_{k=1}^p \overline{b}_{k1}y_{kn}\big\}_{n\in \mathbb{Z}}\cup \big\{\sum_{k=1}^p \overline{b}_{k2}y_{kn}\big\}_{n\in \mathbb{Z}} \cup \dots 
 \cup \big\{\sum_{k=1}^p \overline{b}_{kp}y_{kn}\big\}_{n\in \mathbb{Z}}
\]
where $\big(b_{1k} \,\, b_{2k}\,\,\dots \,b_{pk}\big)^\top$ denotes the $k$-th column of the inverse matrix $M^{-1}$, $k=1, 2, \dots , p$. Moreover, if a sequence as \eqref{lcriesz} with $\det M\neq 0$ is a Riesz basis for $\mathcal{H}$ then the sequence $\{x_n\}_{n\in \mathbb{Z}}$ is also a Riesz basis for $\mathcal{H}$.
\end{lema}
\begin{proof} If $\det M=  0$  then $\{\sum_{k=1}^pa_{1k}x_{km}, \sum_{k=1}^pa_{2k}x_{kn}, \dots, \sum_{k=1}^p a_{pk}x_{kn}\}$ is a linearly dependent system  for each $m \in \mathbb{Z}$, so that the sequence defined in  \eqref{lcriesz}
 is not a Riesz basis.
 
 Conversely, assume that $\det M \neq  0$. Denoting the new sequence defined in \eqref{lcriesz} as  $\{z_n\}_{n\in \mathbb{Z}}$ and its partition as  $\{z_{1n}\}_{n\in \mathbb{Z}}\cup \{z_{2n}\}_{n\in \mathbb{Z}} \cup \dots \cup \{z_{pn}\}_{n\in \mathbb{Z}}$,  we rewrite \eqref{lcriesz} using the  symbolic notation 
 \[ 
 M \big(x_{1n}, x_{2n}, \dots, x_{pn}\big)^\top= \big(z_{1n}, z_{2n}, \dots, z_{pn}\big)^\top\,, \quad n \in \mathbb{Z}\,.
 \]
It is easily deduced what we also symbolically write as
 \[ 
 \big(x_{1n}, x_{2n}, \dots, x_{pn}\big)^\top=  M^{-1}\big(z_{1n}, z_{2n}, \dots, z_{pn}\big)^\top\,, \quad n \in \mathbb{Z} \,.
 \]
Thus $\espan\{z_n \}_{n\in \mathbb{Z}}= \espan\{x_n\}_{n\in \mathbb{Z}}$ and consequently the sequence $\{z_n \}_{n\in \mathbb{Z}}$ is complete in $\mathcal{H}$.
\newline In addition,  for every finite  scalar sequence $\{c_m\}$ and  for every $\{z_m\}= \{\sum_{k=1}^p a_{j_m k}x_{kn_m}\}$ with $n_m \in \mathbb{Z}$ and $j_m \in \{1,2,\dots, p \}$, one has 
\begin{eqnarray*} 
 \big\|\sum_m c_m\,z_m\big\|^2&=&  \Big\|\sum_m c_m\big( \sum_{k=1}^p a_{j_m k}x_{kn_m}\big)\Big\|^2  \le B\sum_m \sum_{k=1}^p |c_m a_{j_m k}|^2 \\ 
 &=& B\sum_m |c_m|^2 \sum_{k=1}^p |a_{j_m k}|^2  \le B D\sum_m |c_m|^2 
 \end{eqnarray*}
 where $B$ is an upper Riesz bound of  $\{x_n\}_{n\in \mathbb{Z}}$ and  $D= \max_{j=1}^p \sum_{k=1}^p |a_{j k}|^2$. Similarly, 
 \begin{eqnarray*} 
 \big\|\sum_m c_m\,z_m\big\|^2&=&  \Big\|\sum_m c_m\big( \sum_{k=1}^p a_{j_m k}x_{kn_m}\big)\Big\|^2  \ge A\sum_m \sum_{k=1}^p |c_m a_{j_m k}|^2 \\ 
 &=& A\sum_m |c_m|^2 \sum_{k=1}^p |a_{j_m k}|^2  \ge AC\sum_m |c_m|^2 
 \end{eqnarray*}
 where $A$ is a lower Riesz bound of the Riesz basis $\{x_n\}_{n\in \mathbb{Z}}$ and  $C= \min_{j=1} ^p\sum_{k=1}^p |a_{j k}|^2$.
 Since $\det M \neq  0$, then $C >0$ and consequently, the sequence $\{z_n\}_{n\in \mathbb{Z}}$ satisfies  the Riesz sequence condition (see \cite[Theorem 3.6.6]{ole:03})
 \[
A'\sum_m |c_m|^2 \le \big\|\sum_m c_m\,z_m\big\|^2 \le B'\sum_m |c_m|^2\,.
\]
 with constants $A'= AC >0$ and $B'= BD>0$.
 
 Finally, since $\langle x_q, y_{q'} \rangle= \delta_{q, q'}$ and $\sum_{k=1}^pa_{jk} b_{kj'}=\delta_{j, j'}$ $$\big\langle\sum_{k=1}^pa_{jk}x_{kn},\sum_{k'=1}^p \overline{b}_{k'j'}y_{k'm}\ \big\rangle= \sum_{k, k'=1}^pa_{jk} b_{k'j'}\langle x_{kn}, y_{k'm} \rangle=\begin{cases} 1 & \text{ if } n \neq m \text{ and } j=j' \\ 
 0 & \text{ otherwise} \end{cases}$$ proving that $\big\{\sum_{k=1}^pa_{jk}x_{kn}\big\}_{j=1, 2, \dots ,p, \ n \in \mathbb{Z}}$ and $\big\{\sum_{k=1}^p \overline{b}_{kj}y_{kn}\big\}_{j=1, 2 , \dots p, \ n \in \mathbb{Z}}$ are a dual pair of Riesz bases for $\mathcal{H}$.
\end{proof}
Observe that in the case of one forward difference, the partition of the Riesz basis  $\{x_n\}_{n\in \mathbb{Z}}$ is $\{x_{2n}\}_{n\in \mathbb{Z}}\cup \{x_{2n+1}\}_{n\in \mathbb{Z}}$, the associated matrix is $M=\begin{pmatrix} 1&0 \\ -1& 1\\ \end{pmatrix}$; its inverse $\begin{pmatrix} 1&0 \\ 1& 1\\ \end{pmatrix}$ gives the dual Riesz basis $\{y_{2n}+y_{2n+1}\}_{n\in \mathbb{Z}} \cup \{y_{2n+1}\}_{n\in \mathbb{Z}}$.

\medskip

Although we will not derive overcomplete expansions in the sequel, it is worth to mention the frame counterpart to the above lemma. Under the circumstances of Lemma \ref{partition} we have:

\begin{lema}
Let $M=\big( a_{jk}\big) \in \mathbb{C}^{q\times p}$ be a matrix of scalars such that $q> p$ and $\text{rank}\,M=p$. By using the rows of $M$ we define
$z_{jn}=\sum_{k=1}^p a_{jk}\,x_{kn}$, $1\le j\le q$ and $n\in \mathbb{Z}$. Then the sequence $\{z_{1n}\}_{n\in \mathbb{Z}}\cup \{z_{2n}\}_{n\in \mathbb{Z}} \cup \dots \cup \{z_{qn}\}_{n\in \mathbb{Z}}$ is a frame for 
$\mathcal{H}$. Moreover, a dual frame $\{\widetilde{z}_{1n}\}_{n\in \mathbb{Z}}\cup \{\widetilde{z}_{2n}\}_{n\in \mathbb{Z}} \cup \dots \cup \{\widetilde{z}_{qn}\}_{n\in \mathbb{Z}}$ is obtained from the columns of a left inverse matrix $N=\big( b_{kj}\big) \in \mathbb{C}^{p\times q}$ of $M$ as $\widetilde{z}_{jn}=\sum_{k=1}^p \overline{b}_{kj}\,y_{kn}$, $1\le j\le q$ and $n\in \mathbb{Z}$. 
\end{lema}
\begin{proof}
Clearly, both sequences are Bessel sequences. Following \cite[Lemma 5.6.2]{ole:03}, it suffices to prove that $x=\sum_{n\in \mathbb{Z}} \sum_{j=1}^q \langle x, \widetilde{z}_{jn}\rangle \,z_{jn}$ for each $x\in \mathcal{H}$. Indeed,
\[
\begin{split}
&\sum_{n\in \mathbb{Z}} \sum_{j=1}^q \big\langle x, \sum_{k=1}^p \overline{b}_{kj}\,y_{kn} \big \rangle \sum_{k=1}^p a_{jk}\,x_{kn}=\sum_{n\in \mathbb{Z}}\sum_{k,k'=1}^p  
\big(\sum_{j=1}^q b_{kj}a_{jk'}\big) \langle x, y_{kn}\rangle x_{k'n}\\
&=\sum_{n\in \mathbb{Z}}\sum_{k,k'=1}^p  \delta_{k,k'} \langle x, y_{kn}\rangle x_{k'n}=\sum_{n\in \mathbb{Z}}\sum_{k=1}^p\langle x, y_{kn}\rangle x_{kn}=\sum_{n\in \mathbb{Z}}\langle x, y_n\rangle x_n=x\,,
\end{split}
\]
since $\sum_{j=1}^q b_{kj}a_{jk'}= \delta_{k,k'}$, and $\{x_n\}_{n\in \mathbb{Z}}$ and $\{y_n\}_{n\in \mathbb{Z}}$ form a pair of dual Riesz bases.
\end{proof}

Observe that the set of all left inverse matrices $N$ of $M$ can be expressed as $N=M^\dag+U\big[\mathbb{I}_q-M\,M^\dag\big]$ where $U$ denotes an arbitrary $p\times q$ matrix and $M^\dag=\big( M^* M\big)^{-1}M^*$ is the Moore-Penrose pseudo-inverse of the matrix $M$.

Observe that the proof of the above lemma is also valid in case the sequences $\{x_n\}_{n\in \mathbb{Z}}$ and $\{y_n\}_{n\in \mathbb{Z}}$ form a pair of dual frames for $\mathcal{H}$.

%%%%%%%%%%%%%%%%%%%%%%%% 
\section{The resulting sampling formulas}
\label{section3}
%%%%%%%%%%%%%%%%%%%%%%%%
Along this section we fix  $a\in [0,1)$ such that $0<\|K_a\|_0\le \|K_a\|_\infty < \infty$, and we will denote 
\[
x_n:=e^{-2\pi i n x} K_a(x)\quad \text{and} \quad y_n:=e^{-2\pi i n x}/\overline{K_a(x)}\,, \quad n\in \mathbb{Z}\,, 
\]
the associated pair of dual Riesz bases for $L^2(0,1)$ which appeared in Section \ref{subsection2-2}. The general procedure is easy: For any $f\in V_\varphi^2$, let $F\in L^2(0,1)$ be such that $\mathcal{T}_\varphi F=f$; once we identify the Riesz basis for $L^2(0,1)$ associated with successive differences considered on the data samples, we should compute its dual Riesz basis  by using Lemma \ref{partition}. Then we expand $F$ with respect to this dual basis in $L^2(0,1)$. Finally, by applying the isomorphism $\mathcal{T}_\varphi$ we obtain a sampling formula in $V_\varphi^2$ with $L^2(\mathbb{R})$-convergence; as $V_\varphi^2$ is a RKHS we obtain the pointwise convergence properties of the sampling expansion: the convergence will be uniform on $\mathbb{R}$. It will be also absolute due the unconditional convergence of an expansion with respect to a Riesz basis.

Next we will exhibit a variety of examples involving different types of samples obtained from {\em forward, backward or central difference operators}, or from {\em average operators}. 

%%%%%%%%%%%%%%%%%%%%%%%%%%%%%%%%
\subsection{The case involving  forward differences}
\label{subsection3-1}
%%%%%%%%%%%%%%%%%%%%%%%%%%%%%%%%
\subsubsection*{Using first and second forward differences}
%%%%%%%%%%%%%%%%%%%%%%%%%%%%%%%%
For the samples $f(a+3n)$, $\Delta_+ f(a+3n)$, and $\Delta^2_+ f(a+3n)$, $n\in \mathbb{Z}$, we have (see \eqref{d+}):
\[
\begin{split}
f(a+3n)&=\big\langle F, e^{-2\pi i (3n) x} K_a \big\rangle_{L^2(0,1)}\,, \\ 
\Delta_+ f(a+3n)&=\big\langle F, e^{-2\pi i (3n+1) x} K_a-e^{-2\pi i (3n) x} K_a\big\rangle_{L^2(0,1)}\,, \\
\Delta^2_+ f(a+3n) &= \big\langle F, e^{-2\pi i (3n+2) x} K_a-2e^{-2\pi i (3n+1) x} K_a+e^{-2\pi i (3n) x} K_a \big\rangle_{L^2(0,1)}\,.
\end{split}
\]
If we consider the partition $\{x_{3n}\}_{n\in \mathbb{Z}}\cup\{x_{3n+1}\}_{n\in \mathbb{Z}} \cup \{x_{3n+2}\}_{n\in \mathbb{Z}}$ of $\{x_n\}_{n\in \mathbb{Z}}$, according to Lemma \ref{partition} the sequence 
\[
\{x_{3n}\}_{n\in \mathbb{Z}} \cup \{x_{3n+1}-x_{3n}\}_{n\in \mathbb{Z}}\cup \{x_{3n+2}-2x_{3n+1}+x_{3n}\}_{n\in \mathbb{Z}}
\]
is a Riesz basis for $L^2(0,1)$ and its dual Riesz basis is 
\[
\{y_{3n}+y_{3n+1}+y_{3n+2}\}_{n\in \mathbb{Z}}\cup \{y_{3n+1}+2y_{3n+2}\}_{n\in \mathbb{Z}}\cup \{y_{3n+2}\}_{n\in \mathbb{Z}}\,,
\]
since
\[
\begin{pmatrix}
1&0&0\\
-1&1&0\\
1&-2&1\\
\end{pmatrix}^{-1}=\begin{pmatrix}
1&0&0\\
1&1&0\\
1&2&1\\
\end{pmatrix}
\]
Expanding $F$ with respect to the above dual basis we obtain
\[
\begin{split}
F=\sum_{n\in \mathbb{Z}} \Big\{&f(a+3n) \big[y_{3n}+y_{3n+1}+y_{3n+2}\big]+\Delta_+ f(a+3n)\big[y_{3n+1}+2y_{3n+2}\big]+\\
&+\Delta^2_+ f(a+3n)\,y_{3n+2}\Big\} \qquad \text{in $L^2(0,1)$}\,.
\end{split}
\]
Applying the isomorphism $\mathcal{T}_\varphi$ and its shifting property \eqref{shifting}, we obtain in $V_\varphi^2$ the sampling formula

\begin{cgris*}{}
{\[
\begin{split}
f(t)=\sum_{n\in \mathbb{Z}}& \Big\{f(a+3n) \big[S_a(t-3n)+S_a(t-3n-1)+S_a(t-3n-2)\big]\\
&+\Delta_+ f(a+3n)\big[S_a(t-3n-1)+2S_a(t-3n-2)\big]\\
&+\Delta^2_+ f(a+3n)\,S_a(t-3n-2)\Big\}\,, \quad t\in \mathbb{R}\,.
\end{split}
\]}
\end{cgris*}
Observe that, due to the shifting property \eqref{shifting}, we have $\mathcal{T}_\varphi \big(y_{pn+k}\big)=S_a\big(t-(pn+k)\big)$.
%%%%%%%%%%%%%%%%%%%%%%%%
\subsubsection*{Using $p-1$ forward differences}
%%%%%%%%%%%%%%%%%%%%%%%%
Proceeding as in the above example, for the samples $f(a+pn)$, $\Delta_+ f(a+pn)$, \dots,  and $\Delta^{p-1}_+ f(a+pn)$, $n\in \mathbb{Z}$, we have (see \eqref{d+}):
{\small
\[
\begin{split}
f(a+pn)&=\big\langle F, e^{-2\pi i (pn) x} K_a \big\rangle_{L^2(0,1)}\,, \\ 
\Delta_+ f(a+pn)&=\big\langle F, e^{-2\pi i (pn+1) x} K_a-e^{-2\pi i (pn) x} K_a\big\rangle_{L^2(0,1)}\,, \\
\Delta^2_+ f(a+pn) &= \big\langle F, e^{-2\pi i (pn+2) x} K_a-2e^{-2\pi i (pn+1) x} K_a+e^{-2\pi i (pn) x} K_a \big\rangle_{L^2(0,1)}\, \\
                       \cdots \cdots &\\ 
\Delta^{p-1}_+ f(a+pn) &= \big\langle F, e^{-2\pi i (pn+p-1) x} K_a-(p-1)e^{-2\pi i (pn+p-2) x} K_a+\binom{p-1}{2}e^{-2\pi i (pn+p-3) x} K_a \\
&+ \dots +(-1)^{p-2}\binom{p-1}{p-2}e^{-2\pi i (pn+1) x} K_a +(-1)^{p-1}e^{-2\pi i (pn) x} K_a \big\rangle_{L^2(0,1)}\, 
\end{split}
\]
}
If we consider the partition $\{x_{pn}\}_{n\in \mathbb{Z}}\cup\{x_{pn+1}\}_{n\in \mathbb{Z}} \cup \dots \cup \{x_{pn+p-1}\}_{n\in \mathbb{Z}}$ of $\{x_n\}_{n\in \mathbb{Z}}$, according to Lemma \ref{partition} the sequence defined by using the rows of the lower triangular matrix:
\[
\begin{pmatrix}
1 &&&&&& \\
-1 & 1 &&&&& \\
1 & -\binom{2}{1} & 1 &&&&\\
-1 & \binom{3}{1} & -\binom{3}{2} & 1&&& \\
\vdots&\vdots&\vdots&\vdots&\ddots&&\\
(-1)^{p-2}&(-1)^{p-3}\binom{p-2}{1}&(-1)^{p-4}\binom{p-2}{2} &(-1)^{p-5}\binom{p-2}{3}& \hdots& 1& \\
(-1)^{p-1}&(-1)^{p-2}\binom{p-1}{1}&(-1)^{p-3}\binom{p-1}{2} &(-1)^{p-4}\binom{p-1}{3}& \hdots& -\binom{p-1}{p-2}& 1\\
\end{pmatrix}
\]
is a Riesz basis for $L^2(0,1)$ and its dual Riesz basis and, consequently, the sampling functions are obtained from the columns of its inverse which is:
\[
\begin{pmatrix}
1 &&&&&& \\
1 & 1 &&&&&  \\
1 & \binom{2}{1} & 1 &&&&\\
1 & \binom{3}{1} & \binom{3}{2} & 1&&& \\
\vdots&\vdots&\vdots&\vdots&\ddots&& \\
1&\binom{p-2}{1}&\binom{p-2}{2}&\binom{p-2}{3}& \hdots & 1& \\
1&\binom{p-1}{1}&\binom{p-1}{2} &\binom{p-1}{3}& \hdots& \binom{p-1}{p-2}& 1\\
\end{pmatrix}
\]
The resulting sampling formula in $V_\varphi^2$ reads:

\begin{cgris*}{}
{\[
\begin{split}
f(t)&=\sum_{n\in \mathbb{Z}}\Big\{f(a+pn) \,T_{a,1}(t-pn)+\Delta_+ f(a+pn)\,T_{a,2}(t-pn)\\ 
&+\Delta^2_+ f(a+pn)\,T_{a,3}(t-pn)+\cdots +\Delta^{p-1}_+ f(a+pn)\,T_{a,p}(t-pn)\Big\}\,, \quad t\in \mathbb{R}\,,
\end{split}
\]}
\end{cgris*}
where, for $1\le k\le p$, the sampling function $T_{a,k}$ is given by
\[
\begin{split}
T_{a,k}(t)&=\binom{k-1}{k-1}S_a(t-(k-1))+\binom{k}{k-1}S_a(t-k)+\\ 
&+\binom{k+1}{k-1}S_a(t-(k+1))+\cdots+\binom{p-1}{k-1}S_a(t-(p-1))\,, \quad t\in \mathbb{R}\,.
\end{split}
\]
Observe that the sampling function $T_{a,k}$ has $p-k+1$ summands, $1\le k\le p$.
%%%%%%%%%%%%%%%%%%%%%%%%
\subsection{The general case involving $p-1$ backward differences}
\label{subsection3-2}
%%%%%%%%%%%%%%%%%%%%%%%%
Here, for the samples $f(a+pn)$, $\Delta_- f(a+pn)$, \dots,  and $\Delta^{p-1}_- f(a+pn)$, $n\in \mathbb{Z}$, we have (see \eqref{d-}):
{\small
\[
\begin{split}
f(a+pn)&=\big\langle F, e^{-2\pi i (pn) x} K_a \big\rangle_{L^2(0,1)}\,, \\ 
\Delta_- f(a+pn)&=\big\langle F, e^{-2\pi i (pn) x} K_a-e^{-2\pi i (pn-1) x} K_a\big\rangle_{L^2(0,1)}\,, \\
\Delta^2_- f(a+pn) &= \big\langle F, e^{-2\pi i (pn) x} K_a-2e^{-2\pi i (pn-1) x} K_a+e^{-2\pi i (pn-2) x} K_a \big\rangle_{L^2(0,1)}\, \\
                               \cdots \cdots &\\ 
\Delta^{p-1}_- f(a+pn) &= \big\langle F, e^{-2\pi i (pn) x} K_a-(p-1)e^{-2\pi i (pn-1) x} K_a+\binom{p-1}{2}e^{-2\pi i (pn-2) x} K_a \\
&+ \dots +(-1)^{p-2}\binom{p-1}{p-2}e^{-2\pi i (pn-p+2) x} K_a +(-1)^{p-1}e^{-2\pi i (pn-p+1) x} K_a \big\rangle_{L^2(0,1)}\, 
\end{split}
\]
}
If we consider the partition $\{x_{pn-p+1}\}_{n\in \mathbb{Z}}\cup\{x_{pn-p+2}\}_{n\in \mathbb{Z}} \cup \dots \cup \{x_{pn-1}\}_{n\in \mathbb{Z}} \cup \{x_{pn}\}_{n\in \mathbb{Z}}$ of $\{x_n\}_{n\in \mathbb{Z}}$, according to Lemma \ref{partition} the sequence defined by using the rows of the upper triangular matrix: 
\[
\begin{pmatrix}
(-1)^{p-1} &(-1)^{p-2}\binom{p-1}{1}&(-1)^{p-3}\binom{p-1}{2}&\hdots&-\binom{p-1}{3}&\binom{p-1}{2}&-\binom{p-1}{1}& 1\\
 & (-1)^{p-2} &(-1)^{p-3}\binom{p-2}{1}&\hdots&-\binom{p-2}{3}&\binom{p-2}{2}&-\binom{p-2}{1}&1 \\
 &&(-1)^{p-3}&\hdots&-\binom{p-3}{3}&\binom{p-3}{2}&-\binom{p-3}{1}&1\\
&&&\ddots&\vdots&\vdots&\vdots&\vdots\\
&& &&-1&\binom{3}{2}&-\binom{3}{1}& 1\\
&&&&&1&-\binom{2}{1}&1\\
&&&&&&-1&1 \\
&&&&& & & 1\\
\end{pmatrix}
\]
is a Riesz basis for $L^2(0,1)$ and its dual Riesz basis and, consequently, the sampling functions are obtained from the columns of its inverse which here coincides with itself. Therefore, writing from the last column to the first one, the resulting sampling formula in $V_\varphi^2$ reads:

\begin{cgris*}{}
{
\[
\begin{split}
f(t)&=\sum_{n\in \mathbb{Z}}\Big\{f(a+pn) \,\widetilde{T}_{a,1}(t-pn)+\Delta_- f(a+pn)\,\widetilde{T}_{a,2}(t-pn)\\ 
&+\Delta^2_- f(a+pn)\,\widetilde{T}_{a,3}(t-pn)+\cdots +\Delta^{p-1}_- f(a+pn)\,\widetilde{T}_{a,p}(t-pn)\Big\}\,, \quad t\in \mathbb{R}\,,
\end{split}
\]}
\end{cgris*}
where, for $1\le k\le p$, the sampling function $\widetilde{T}_{a,k}$ is given by
\[
\begin{split}
\widetilde{T}_{a,k}(t)&=(-1)^{k-1}\Big[\binom{p-1}{k-1}S_a(t+p-1)+\binom{p-2}{k-1}S_a(t+p-2)+\\
&+\binom{p-3}{k-1}S_a(t+p-3)+\cdots+\binom{k-1}{k-1}S_a(t+k-1)\Big]\,, \quad t\in \mathbb{R}\,.
\end{split}
\]
Observe that the sampling function $\widetilde{T}_{a,k}$ has $p-k+1$ summands, $1\le k\le p$.
%%%%%%%%%%%%%%%%%%%%%%%%
\subsection{Two representative examples involving forward and backward differences}
\label{subsection3-3}
%%%%%%%%%%%%%%%%%%%%%%%%%%%%%%%%%
\subsubsection*{Using first forward and backward differences}
%%%%%%%%%%%%%%%%%%%%%%%%%%%%%%%%%
For the samples $f(a+3n)$, $\Delta_- f(a+3n)$, and $\Delta_+ f(a+3n)$, $n\in \mathbb{Z}$, we have:
\[
\begin{split}
\Delta_- f(a+3n)&=\big\langle F, e^{-2\pi i (3n) x} K_a-e^{-2\pi i (3n-1) x} K_a \big\rangle_{L^2(0,1)}\,, \\
f(a+3n)&=\big\langle F, e^{-2\pi i (3n) x} K_a \big\rangle_{L^2(0,1)}\,, \\ 
\Delta_+ f(a+3n)&=\big\langle F, e^{-2\pi i (3n+1) x} K_a-e^{-2\pi i (3n) x} K_a\big\rangle_{L^2(0,1)}\,.
\end{split}
\]
If we consider the partition $\{x_{3n-1}\}_{n\in \mathbb{Z}}\cup\{x_{3n}\}_{n\in \mathbb{Z}} \cup \{x_{3n+1}\}_{n\in \mathbb{Z}}$ of $\{x_n\}_{n\in \mathbb{Z}}$, according to Lemma \ref{partition} the sequence 
\[
\{-x_{3n-1}+x_{3n}\}_{n\in \mathbb{Z}} \cup \{x_{3n}\}_{n\in \mathbb{Z}}\cup \{-x_{3n}+x_{3n+1}\}_{n\in \mathbb{Z}}
\] 
is a Riesz basis for $L^2(0,1)$ and its dual Riesz basis is
\[
\{-y_{3n-1}\}_{n\in \mathbb{Z}} \cup \{y_{3n-1}+y_{3n}+y_{3n+1}\}_{n\in \mathbb{Z}}\cup \{y_{3n+1}\}_{n\in \mathbb{Z}}
\]
since
\[
\begin{pmatrix}
-1&1&0\\
0&1&0\\
0&-1&1\\
\end{pmatrix}^{-1}=\begin{pmatrix}
-1&1&0\\
0&1&0\\
0&1&1\\
\end{pmatrix}
\]
The resulting sampling formula in $V_\varphi^2$ reads:

\begin{cgris*}{}
{
\[
\begin{split}
f(t)=\sum_{n\in \mathbb{Z}}\Big\{&-\Delta_- f(a+3n)\, S_a(t-3n+1)\\
&+f(a+3n) \big[S_a(t-3n+1)+S_a(t-3n)+S_a(t-3n-1)\big]\\
&+ \Delta_+ f(a+3n)\,S_a(t-3n-1)\Big\}\,,\quad t\in \mathbb{R}\,.
\end{split}
\]}
\end{cgris*}
%%%%%%%%%%%%%%%%%%%%%%%%
\subsubsection*{Using first backward and first and second forward differences}
%%%%%%%%%%%%%%%%%%%%%%%%
For the samples $f(a+4n)$, $\Delta_+ f(a+4n)$, $\Delta^2_+ f(a+4n)$, and $\Delta_- f(a+4n)$, $n\in \mathbb{Z}$, we have (see \eqref{d+} and \eqref{d-}):
\[
\begin{split}
\Delta_- f(a+4n)&=\big\langle F, e^{-2\pi i (4n) x} K_a-e^{-2\pi i (4n-1) x} K_a \big\rangle_{L^2(0,1)}\,, \\
f(a+4n)&=\big\langle F, e^{-2\pi i (4n) x} K_a \big\rangle_{L^2(0,1)}\,, \\ 
\Delta_+ f(a+4n)&=\big\langle F, e^{-2\pi i (4n+1) x} K_a-e^{-2\pi i (4n) x} K_a\big\rangle_{L^2(0,1)}\,,\\
\Delta^2_+ f(a+4n)&=\big\langle F, e^{-2\pi i (4n+2) x} K_a-2e^{-2\pi i (4n+1) x} K_a+e^{-2\pi i (4n) x} K_a \big\rangle_{L^2(0,1)}\,.
\end{split}
\]
If we consider the partition $\{x_{4n-1}\}_{n\in \mathbb{Z}}\cup\{x_{4n}\}_{n\in \mathbb{Z}} \cup \{x_{4n+1}\}_{n\in \mathbb{Z}}\cup \{x_{4n+2}\}_{n\in \mathbb{Z}}$ of $\{x_n\}_{n\in \mathbb{Z}}$, according to Lemma \ref{partition} the sequence 
\[
\{-x_{4n-1}+x_{4n}\}_{n\in \mathbb{Z}}\cup \{x_{4n}\}_{n\in \mathbb{Z}}\cup \{-x_{4n}+x_{4n+1}\}_{n\in \mathbb{Z}}
\cup \{ x_{4n}-2x_{4n+1}+x_{4n+2}\}_{n\in \mathbb{Z}}
\] 
is a Riesz basis for $L^2(0,1)$ and its dual Riesz basis is
\[
\{-y_{4n-1}\}_{n\in \mathbb{Z}}\cup \{y_{4n-1}+y_{4n}+y_{4n+1}+y_{4n+2}\}_{n\in \mathbb{Z}}
\cup \{y_{4n+1}+2y_{4n+2}\}_{n\in \mathbb{Z}}\cup\{y_{4n+2}\}_{n\in \mathbb{Z}}
\]
since
\[
\begin{pmatrix}
-1&1&0&0\\
0&1&0&0\\
0&-1&1&0\\
0&1&-2&1\\
\end{pmatrix}^{-1}=\begin{pmatrix}
-1&1&0&0\\
0&1&0&0\\
0&1&1&0\\
0&1&2&1\\
\end{pmatrix}
\]
The resulting sampling formula in $V_\varphi^2$ reads:

\begin{cgris*}{}
{
\[
\begin{split}
f(t)=\sum_{n\in \mathbb{Z}}\Big\{&-\Delta_- f(a+4n)S_a(t-4n+1)+f(a+4n) \big[S_a(t-4n+1)+S_a(t-4n)\\
&+S_a(t-4n-1)+S_a(t-4n-2)\big]\\
&+ \Delta_+ f(a+4n)\big[S_a(t-4n-1)+2S_a(t-4n-2) \big]\\
&+\Delta^2_+ f(a+4n)\,S_a(t-4n-2)\Big\}\,,\quad t\in \mathbb{R}\,.
\end{split}
\]
}
\end{cgris*}

From the above example and Sections \ref{subsection3-1} and \ref{subsection3-2} the general case involving $p$ forward and $q-1$ backward differences with sampling period $p+q$ easily comes out.
%%%%%%%%%%%%%%%%%%%%%%%%%%%%%
\subsection{Some examples involving averages and differences}
\label{subsection3-4}
%%%%%%%%%%%%%%%%%%%%%%%%%%%%%
\subsubsection*{Using forward averages and differences}
%%%%%%%%%%%%%%%%%%%%%%%%%%%%%
For the samples $\mu_+ f(a+2n)$ and $\Delta_+ f(a+2n)$, $n\in \mathbb{Z}$, we have:
\[
\begin{split}
\mu_+ f(a+2n)&=\big\langle F, \frac{1}{2}(e^{-2\pi i (2n+1) x} K_a+e^{-2\pi i (2n) x} K_a \big\rangle_{L^2(0,1)}\,, \\
\Delta_+ f(a+2n)&=\big\langle F, e^{-2\pi i (2n+1) x} K_a-e^{-2\pi i (2n) x} K_a\big\rangle_{L^2(0,1)}\,.
\end{split}
\]
If we consider the partition $\{x_{2n}\}_{n\in \mathbb{Z}}\cup\{x_{2n+1}\}_{n\in \mathbb{Z}}$ of $\{x_n\}_{n\in \mathbb{Z}}$, according to Lemma \ref{partition} the sequence 
\[
\{\frac{1}{2}x_{2n}+\frac{1}{2}x_{2n+1}\}_{n\in \mathbb{Z}} \cup \{-x_{2n}+x_{2n+1}\}_{n\in \mathbb{Z}}
\] 
is a Riesz basis for $L^2(0,1)$ and its dual Riesz basis is
\[
\{y_{2n}+y_{2n+1}\}_{n\in \mathbb{Z}} \cup \{-\frac{1}{2}y_{2n}+\frac{1}{2}y_{2n+1}\}_{n\in \mathbb{Z}}
\]
since
\[
\begin{pmatrix}
1/2&1/2\\
-1&1\\
\end{pmatrix}^{-1}=\begin{pmatrix}
1&-1/2\\
1&1/2\\
\end{pmatrix}
\]
The resulting sampling formula in $V_\varphi^2$ reads:

\begin{cgris*}{}
{
\[
\begin{split}
f(t)=\sum_{n\in \mathbb{Z}}&\Big\{\mu_+ f(a+2n)\big[S_a(t-2n)+S_a(t-2n-1)\big] \\
&+ \Delta_+ f(a+2n)\frac{1}{2}\big[-S_a(t-2n)+S_a(t-2n-1)\big]\Big\}\,,\quad t\in \mathbb{R}\,.
\end{split}
\]}
\end{cgris*}
%%%%%%%%%%%%%%%%%%%%%%%%%%%%%
\subsubsection*{Using central averages and differences} 
%%%%%%%%%%%%%%%%%%%%%%%%%%%%%
For the samples $f(a+3n)$, $\mu_0 f(a+3n)$, and $\Delta_0 f(a+3n)$, $n\in \mathbb{Z}$, we have:
\[
\begin{split}
\mu_0 f(a+3n)&=\big\langle F, \frac{1}{2}(e^{-2\pi i (3n+1) x} K_a+e^{-2\pi i (3n-1) x} K_a \big\rangle_{L^2(0,1)}\,, \\
f(a+3n)&=\big\langle F, e^{-2\pi i (3n) x} K_a \big\rangle_{L^2(0,1)}\,, \\ 
\Delta_0 f(a+3n)&=\big\langle F, e^{-2\pi i (3n+1) x} K_a-e^{-2\pi i (3n-1) x} K_a\big\rangle_{L^2(0,1)}\,.
\end{split}
\]
If we consider the partition $\{x_{3n-1}\}_{n\in \mathbb{Z}}\cup\{x_{3n}\}_{n\in \mathbb{Z}} \cup \{x_{3n+1}\}_{n\in \mathbb{Z}}$ of $\{x_n\}_{n\in \mathbb{Z}}$, according to Lemma \ref{partition} the sequence 
\[
\{x_{3n}\}_{n\in \mathbb{Z}} \cup \{\frac{1}{2}x_{3n-1}+\frac{1}{2}x_{3n+1}\}_{n\in \mathbb{Z}}\cup \{-x_{3n-1}+x_{3n+1}\}_{n\in \mathbb{Z}}
\] 
is a Riesz basis for $L^2(0,1)$ and its dual Riesz basis is
\[
\{y_{3n}\}_{n\in \mathbb{Z}} \cup \{y_{3n-1}+y_{3n+1}\}_{n\in \mathbb{Z}}\cup \{-\frac{1}{2}y_{3n-1}+\frac{1}{2}y_{3n+1}\}_{n\in \mathbb{Z}}
\]
since
\[
\begin{pmatrix}
0&1&0\\
1/2&0&1/2\\
-1&0&1\\
\end{pmatrix}^{-1}=\begin{pmatrix}
0&1&-1/2\\
1&0&0\\
0&1&1/2\\
\end{pmatrix}
\]
The resulting sampling formula in $V_\varphi^2$ reads:

\begin{cgris*}{}
{
\[
\begin{split}
f(t)=\sum_{n\in \mathbb{Z}}&\Big[f(a+3n)\, S_a(t-3n)+\mu_0 f(a+3n) \big[S_a(t-3n+1)+S_a(t-3n-1)\big]\\
&+ \Delta_0 f(a+3n)\frac{1}{2}\big[-S_a(t-3n+1)+S_a(t-3n-1)\big]\Big]\,,\quad t\in \mathbb{R}\,.
\end{split}
\]}
\end{cgris*}
%%%%%%%%%%%%%%%%%%%%%%%%
\section{Sampling results for the two-dimensional case}
\label{section4}
%%%%%%%%%%%%%%%%%%%%%%%%
Firstly, we recall the usual sampling formulas in the two-dimensional setting:
%%%%%%%%%%%%%%%%%%%%%%%%%%
\subsection{Sampling formulas for two-dimensional shift-invariant subspaces}
\label{subsectio4-1}
%%%%%%%%%%%%%%%%%%
\subsubsection*{The general two-dimensional case}
%%%%%%%%%%%%%%%%%%
Consider a generator $\Phi \in L^2(\mathbb{R}^2)$ such that the sequence $\big\{\Phi (t-n,s-m)\big\}_{n,m\in \mathbb{Z}}$ forms a Riesz sequence for $L^2(\mathbb{R}^2)$. Thus,
\[
V^2_\Phi:=\overline{\espan}\{\Phi (t-n,s-m)\}_{n,m\in \mathbb{Z}}=\Big\{\sum_{n,m\in \mathbb{Z}} a_{nm}\,\Phi(t-n,s-m)\,\,:\,\, \{a_{nm}\}\in \ell^2(\mathbb{Z}^2) \Big\}\,.
\]
Assuming that $\Phi$ is continuous in $\mathbb{R}^2$ and that $\sum_{n,m\in \mathbb{Z}} |\Phi (t-n,s-m)|^2$ is bounded in $[0,1]^2$ one can easily deduce the same results that in Section \ref{section2} for the one dimensional case. For instance, the isomorphism
$\mathcal{T}_\Phi: L^2(0,1)^2 \longrightarrow V^2_\Phi$ which maps the orthonormal basis $\{e^{-2\pi i n x}\,e^{-2\pi i m y}\}_{n,m\in \mathbb{Z}}$ onto the Riesz basis 
$\{\Phi(t-n,s-m)\}_{n,m\in \mathbb{Z}}$ satisfies for any $F\in L^2(0,1)^2$ the shifting property
\[
\mathcal{T}_\Phi \big[e^{-2\pi i n x} e^{-2\pi i m y} F(x,y)\big](t, s)=\mathcal{T}_\Phi[F](t-n,s-m)\,, \quad t, s\in \mathbb{R}\,,\,n,m\in \mathbb{Z}\,.
\]
Moreover, any $f=\mathcal{T}_\Phi F \in V^2_\varphi$ can be expressed as
\begin{equation}
\label{expression3}
f(t,s)=\Big\langle F(x,y), K_{t,s}(x,y) \Big\rangle_{L^2(0,1)^2}\,, \quad t, s\in \mathbb{R}\,,
\end{equation}
where  $K_{t,s}(x,y)=\sum_{n,m\in \mathbb{Z}} \overline{\Phi(t-n, s-m)}\, e^{-2\pi i nx}e^{-2\pi i my}\in L^2(0,1)^2$. The details can be found in \cite{garcia:09}.
Thus, for fixed $a, b\in [0,1)$  and $n, m \in \mathbb{Z}$ we get the expression for the samples
\[
f(a+n, b+m)=\big\langle F, K_{a+n,b+m}\big\rangle_{L^2(0,1)^2} = \big\langle F, {\rm e}^{-2\pi i n x}{\rm e}^{-2\pi i m y}K_{a,b}(x,y)\big\rangle_{L^2(0,1)^2}\,.
\]
The sequence $\big\{{\rm e}^{-2\pi i n x}\,{\rm e}^{-2\pi i my}K_{a,b}(x,y)\big\}_{n,m\in \mathbb{Z}}$ is a Riesz basis for $L^2(0,1)^2$ if and only if the inequalities $0<\|K_{a,b}\|_0\le \|K_{a,b}\|_\infty < \infty$ hold.  Besides, its dual Riesz basis is the sequence $\big\{{\rm e}^{-2\pi i n x}\,{\rm e}^{-2\pi i my}/\overline{K_{a,b}}(x,y)\big\}_{n,m\in \mathbb{Z}}$.

Proceeding as in Section \ref{subsection2-2} we derive for any $f\in V^2_\Phi$ the sampling formula
\begin{equation}
\label{sf3}
f(t,s)=\sum_{n,m\in \mathbb{Z}} f(a+n, b+m)\,S_{a,b}(t-n,s-m)\,, \quad t, s\in \mathbb{R}\,,
\end{equation}
where $S_{a,b}:=\mathcal{T}_\Phi\big(1/\overline{K_{a,b}}\big)\in V_\Phi^2$. The  convergence of the series is absolute due to the unconditional convergence of a Riesz basis expansion and it is also uniform on $\mathbb{R}^2$ since the shift-invariant space $V_\Phi^2$ is a RKHS contained in $L^2(\mathbb{R}^2)$.

%%%%%%%%%%%%%%%%%%%%%%%%%
\subsubsection*{The case of two separable variables}
%%%%%%%%%%%%%%%%%%%%%%%%%
Consider two generators $\varphi(x)$ and $\psi(y)$ in $L^2(\mathbb{R})$ satisfying the same hypotheses as in Section \ref{section2}, and  define $\Phi(x,y):=\varphi(x)\psi(y)$ in $L^2(\mathbb{R}^2)$ and its associated shift-invariant subspace $V^2_{\varphi\psi}:=V^2_\Phi$. It can be seen as the tensor product $V^2_\varphi \otimes V^2_\psi$. For a brief on tensor products see, for instance, \cite{garcia:16} and references therein. 

Straightforward, $\mathcal{T}_{\varphi\psi}:=\mathcal{T}_\Phi=\mathcal{T}_\varphi \otimes \mathcal{T}_\psi$ holds and, consequently, for $H, G\in L^2(0,1)$ and $n, m \in \mathbb{Z}$, the shifting property reads:
\[
\mathcal{T}_{\varphi\psi} \big[e^{-2\pi i n x} e^{-2\pi i m y} H(x)G(y)\big](t, s)=\mathcal{T}_\varphi[H](t-n)\mathcal{T}_\psi[G](s-m)\,, \quad t, s\in \mathbb{R}\,.
\] 
Besides, any $f=\mathcal{T}_{\varphi\psi} F \in V^2_{\varphi\psi}$, where $F\in L^2(0,1)^2$, can be expressed as
\begin{equation}
\label{expression2}
f(t,s)=\Big\langle F(x,y), K_t(x)\widetilde{K}_s(y) \Big\rangle_{L^2(0,1)^2}\,, \quad t, s\in \mathbb{R}\,,
\end{equation}
where  $K_t(x)=\sum_{n\in \mathbb{Z}} \overline{\varphi(t-n)}\, e^{-2\pi i nx}$ and $\widetilde{K}_s(y)=\sum_{m\in \mathbb{Z}} \overline{\psi(s-m)}\, e^{-2\pi i my}$ belong to $L^2(0,1)$. Thus, for fixed $a, b\in [0,1)$ and $n, m \in \mathbb{Z}$ we get the following expression for the samples
\[
f(a+n, b+m)=\big\langle F, K_{a+n}\,\widetilde{K}_{b+m}\big\rangle_{L^2(0,1)^2} = \big\langle F, {\rm e}^{-2\pi i n x}K_a(x)\, {\rm e}^{-2\pi i m y}\widetilde{K}_b(y)\big\rangle_{L^2(0,1)^2}\,.
\]
Assuming that $0<\|K_a\|_0\le \|K_a\|_\infty < \infty$ and $0<\|\widetilde{K}_b\|_0\le \|\widetilde{K}_b\|_\infty < \infty$, the sequence $\big\{{\rm e}^{-2\pi i n x}{\rm e}^{-2\pi i m y}K_a(x)\widetilde{K}_b(y) \big\}_{n,m\in \mathbb{Z}}$ is a Riesz basis for $L^2(0,1)^2$ with dual basis 
$\big\{{\rm e}^{-2\pi i n x}{\rm e}^{-2\pi i m y}/(\overline{K_a(x)}\overline{\widetilde{K}_b(y)})\big\}_{n,m\in \mathbb{Z}}$. By using the notation in Section \ref{section3} the sequences $\big\{x_n\otimes \widetilde{x}_m\big\}_{n,m\in \mathbb{Z}}$ and $\big\{y_n\otimes \widetilde{y}_m\big\}_{n,m\in \mathbb{Z}}$ form a pair of dual Riesz bases in $L^2(0,1)^2$. 

Proceeding as in Section \ref{subsection2-2} we derive for any $f\in V^2_{\varphi\psi}$ the sampling formula
\begin{equation}
\label{sf2}
f(t,s)=\sum_{n,m\in \mathbb{Z}} f(a+n, b+m)\,S_a(t-n)\,\widetilde{S}_b(s-m)\,, \quad t, s\in \mathbb{R}\,,
\end{equation}
where $S_a:=\mathcal{T}_\varphi\big(1/\overline{K_a}\big)\in V_\varphi^2$ and $\widetilde{S}_b:=\mathcal{T}_\psi\big(1/\overline{\widetilde{K}_b}\big)\in V_\psi^2$. The  convergence of the series is absolute due to the unconditional convergence of a Riesz basis expansion and it is also uniform on $\mathbb{R}^2$ since the shift-invariant space $V_{\varphi\psi}^2$ is a RKHS contained in $L^2(\mathbb{R}^2)$.

%%%%%%%%%%%%%%%%%%%%%%%%
\subsection{The resulting sampling formulas}
\label{subsection4-2}
%%%%%%%%%%%%%%%%%%%%%%%%
Throughout this section we will use the  {\em forward differences} defined as:
\[
\Delta_+^{k,k'}f=\Delta_+^{1,0} \big(\Delta_+^{k-1,k'}f\big)=\Delta_+^{0,1}\big(\Delta_+^{k,k'-1}f \big)\,,
\]
where
\[
\Delta_+^{1,0}f(t,s)=f(t+1,s)-f(t,s)\quad \text{and}\quad \Delta_+^{0,1}f(t,s)=f(t,s+1)-f(t,s)\,;
\]
we adopt the convention $\Delta_+^{0,0}=I$. 

For any $f\in V^2_{\varphi\psi}$ (such that $\mathcal{T}_{\varphi\psi}F=f$) or $f\in V^2_\Phi$ (such that $\mathcal{T}_\Phi F=f$), we respectively get
\begin{equation}
\label{d+1}
\Delta_+^{k,k'} f(t,s)=\big\langle F, (e^{-2\pi ix}-1)^k\,(e^{-2\pi iy}-1)^{k'}\,K_t(x)\,\widetilde{K}_s(y) \big\rangle_{L^2(0,1)^2}\,, 
\end{equation}
or
\begin{equation}
\label{d+2}
\Delta_+^{k,k'} f(t,s)=\big\langle F, (e^{-2\pi ix}-1)^k\,(e^{-2\pi iy}-1)^{k'}\,K_{t,s}(x,y)\big\rangle_{L^2(0,1)^2}\,.
\end{equation}

Next we exhibit a couple of examples involving forward differences up to second order. The obtained results could be generalized to high order forward differences and even to higher dimensions; besides, backward differences or averages could be used as in the previous section.
%%%%%%%%%%%%%%%%%%%%%%%%%%%%%%%%%%%%%%%%%%%%
\subsubsection*{Using the samples $\Delta_+^{k,k'}f(a+2n,b+3m)$, ($k=0,1\,;\, k'=0,1,2$)  in $V^2_{\varphi\psi}$}
%%%%%%%%%%%%%%%%%%%%%%%%%%%%%%%
For any $f=\mathcal{T}_{\varphi\psi}F$ in $V^2_{\varphi\psi}$ we have the expression for the samples (see \eqref{d+1}):
\[
\Delta_+^{k,k'} f(a+2n,b+3m)=\Big\langle F, (e^{-2\pi ix}-1)^k(e^{-2\pi iy}-1)^{k'}e^{-2\pi i (2n)x}K_a(x)e^{-2\pi i (3m)y}\widetilde{K}_b(y) \Big\rangle\,,
\]
where $n,m\in \mathbb{Z}$ and $k=0,1\,;\, k'=0,1,2$. 

Consider the Riesz basis $\big\{x_n\otimes \widetilde{x}_m\big\}_{n,m\in \mathbb{Z}}$ for $L^2(0,1)^2$, where $x_n:=e^{-2\pi i n x} K_a(x)$ and $\widetilde{x}_m:=e^{-2\pi i m y} \widetilde{K}_b(y)$, $n,m\in \mathbb{Z}$.  According to Lemma \ref{partition}, for the partition 
\[
\begin{split}
&\{x_{2n}\otimes  \widetilde{x}_{3m}\}_{n,m\in \mathbb{Z}}\cup \{x_{2n}\otimes  \widetilde{x}_{3m+1}\}_{n,m\in \mathbb{Z}}\cup \{x_{2n}\otimes  \widetilde{x}_{3m+2}\}_{n,m\in \mathbb{Z}}\\
\cup\,\,&\{x_{2n+1}\otimes  \widetilde{x}_{3m}\}_{n,m\in \mathbb{Z}}
\cup \{x_{2n+1}\otimes  \widetilde{x}_{3m+1}\}_{n,m\in \mathbb{Z}}\cup \{x_{2n+1}\otimes  \widetilde{x}_{3m+2}\}_{n,m\in \mathbb{Z}}\,, 
\end{split}
\]
the  matrix of the Riesz basis associated with the above samples will be the {\em Kronecker  product} (tensor product) of the corresponding matrices derived, for the one-dimensional case, in Section \ref{section3}. That is:
\[
\begin{pmatrix} 1&0 \\ -1& 1\\ \end{pmatrix} \otimes \begin{pmatrix}
1&0&0\\
-1&1&0\\
1&-2&1\\
\end{pmatrix}=
\begin{pmatrix}
1&0&0&0&0&0\\
-1&1&0&0&0&0\\
1&-2&1&0&0&0\\
-1&0&0&1&0&0\\
1&-1&0&-1&1&0\\
-1&2&-1&1&-2&1\\
\end{pmatrix}\,.
\]
Its inverse is the Kronecker product of their respective inverses (see, for instance, Ref. \cite{laub:05}); namely:
\begin{equation}
\label{i1}
\begin{pmatrix} 1&0 \\ 1& 1\\ \end{pmatrix} \otimes \begin{pmatrix}
1&0&0\\
1&1&0\\
1&2&1\\
\end{pmatrix}=
\begin{pmatrix}
1&0&0&0&0&0\\
1&1&0&0&0&0\\
1&2&1&0&0&0\\
1&0&0&1&0&0\\
1&1&0&1&1&0\\
1&2&1&1&2&1\\
\end{pmatrix}\,.
\end{equation}
Thus, the corresponding sampling formula for each $f\in V^2_{\varphi\psi}$ reads:

\begin{cgris*}{}
{
\[
f(t,s)=\sum_{n,m\in \mathbb{Z}}\Big\{\sum_{k=0}^1\sum_{k'=0}^2 \,\Delta_+^{k,k'}f(a+2n,b+3m)\,\widetilde{T}^{k,k'}_{a,b}(t-2n,s-3m)\Big\}\,, \quad t,s\in \mathbb{R}\,,
\]}
\end{cgris*}
where the sampling functions $\widetilde{T}^{k,k'}_{a,b}(t,s)$ are obtained from the columns of the inverse matrix \eqref{i1} as
\[
\begin{split}
\widetilde{T}^{0,0}_{a,b}(t,s)&=S_a(t)\,\widetilde{S}_b(s)+S_a(t)\,\widetilde{S}_b(s-1)+S_a(t)\,\widetilde{S}_b(s-2)+S_a(t-1)\,\widetilde{S}_b(s)\\
&+S_a(t-1)\,\widetilde{S}_b(s-1)+S_a(t-1)\,\widetilde{S}_b(s-2)\,;\\
\widetilde{T}^{0,1}_{a,b}(t,s)&=S_a(t)\,\widetilde{S}_b(s-1)+2S_a(t)\,\widetilde{S}_b(s-2)+S_a(t-1)\,\widetilde{S}_b(s-1)+2S_a(t-1)\,\widetilde{S}_b(s-2)\,;\\
\widetilde{T}^{0,2}_{a,b}(t,s)&=S_a(t)\,\widetilde{S}_b(s-2)+S_a(t-1)\,\widetilde{S}_b(s-2)\,;\\
\widetilde{T}^{1,0}_{a,b}(t,s)&=S_a(t-1)\,\widetilde{S}_b(s)+S_a(t-1)\,\widetilde{S}_b(s-1)+S_a(t-1)\,\widetilde{S}_b(s-2)\,; \\
\widetilde{T}^{1,1}_{a,b}(t,s)&=S_a(t-1)\,\widetilde{S}_b(s-1)+2S_a(t-1)\,\widetilde{S}_b(s-2)\,; \\
\widetilde{T}^{1,2}_{a,b}(t,s)&=S_a(t-1)\,\widetilde{S}_b(s-2)\,, \quad t,s \in \mathbb{R}\,.
\end{split}
\]

Observe that $\mathcal{T}_{\varphi\psi}\big(y_{2n+k}\otimes \widetilde{y}_{3m+k'}\big)=S_a(t-(2n+k))\widetilde{S}_b(s-(3m+k'))$, where $y_n:=e^{-2\pi i n x}/\overline{K_a(x)}$ and $\widetilde{y}_m:=e^{-2\pi i m y}/\overline{\widetilde{K}_b(y)}$, $n,m\in \mathbb{Z}$.
%%%%%%%%%%%%%%%%%%%%%%%%%%%%%%%%%%%%%%%%%%%%
\subsubsection*{Using the samples $\Delta_+^{k,k'}f(a+3n,b+3m)$, ($k,k'=0,1,2$)  in $V^2_\Phi$}
%%%%%%%%%%%%%%%%%%%%%%%%%%%%%%%
For any $f=\mathcal{T}_\Phi F$ in $V^2_{\Phi}$ we have the expression for the samples (see \eqref{d+2}):
\[
\Delta_+^{k,k'} f(a+3n,b+3m)=\Big\langle F, (e^{-2\pi ix}-1)^k(e^{-2\pi iy}-1)^{k'}e^{-2\pi i (3n)x}\,e^{-2\pi i (3m)y}K_{a,b}(x,y) \Big\rangle\,,
\]
where $n,m\in \mathbb{Z}$ and $k,k'=0,1,2$. Consider the Riesz basis $\big\{x_{n,m}\big\}_{n,m\in \mathbb{Z}}$ for $L^2(0,1)^2$, where $x_{n,m}:={\rm e}^{-2\pi i n x}\,{\rm e}^{-2\pi i my}K_{a,b}(x,y)$, $n,m\in \mathbb{Z}$. According to Lemma \ref{partition}, for the partition 
\[
\begin{split}
&\{x_{3n,3m}\}_{n,m\in \mathbb{Z}}\cup \{x_{3n,3m+1}\}_{n,m\in \mathbb{Z}}\cup\{x_{3n,3m+2}\}_{n,m\in \mathbb{Z}}\\
\cup\,\,&\{x_{3n+1,3m}\}_{n,m\in \mathbb{Z}}\cup \{x_{3n+1,3m+1}\}_{n,m\in \mathbb{Z}}\cup \{x_{3n+1,3m+2}\}_{n,m\in \mathbb{Z}}\\
\cup\,\,&\{x_{3n+2,3m}\}_{n,m\in \mathbb{Z}}\cup \{x_{3n+2,3m+1}\}_{n,m\in \mathbb{Z}}\cup \{x_{3n+2,3m+2}\}_{n,m\in \mathbb{Z}},
\end{split}
\]
the Riesz basis associated to the above samples has as matrix the Kronecker product: 
\[
\begin{pmatrix}
1&0&0\\
-1&1&0\\
1&-2&1\\
\end{pmatrix} \otimes \begin{pmatrix}
1&0&0\\
-1&1&0\\
1&-2&1\\
\end{pmatrix}=\begin{pmatrix}
1&0&0&0&0&0&0&0&0\\
-1&1&0&0&0&0&0&0&0\\
1&-2&1&0&0&0&0&0&0\\
-1&0&0&1&0&0&0&0&0\\
1&-1&0&-1&1&0&0&0&0\\
-1&2&-1&1&-2&1&0&0&0\\
1&0&0&-2&0&0&1&0&0\\
-1&1&0&2&-2&0&-1&1&0\\
1&-2&1&-2&4&-2&1&-2&1\\
\end{pmatrix}\,,
\]
whose inverse is:
\begin{equation}
\label{i2}
\begin{pmatrix}
1&0&0\\
1&1&0\\
1&2&1\\
\end{pmatrix} \otimes \begin{pmatrix}
1&0&0\\
1&1&0\\
1&2&1\\
\end{pmatrix}=\begin{pmatrix}
1&0&0&0&0&0&0&0&0\\
1&1&0&0&0&0&0&0&0\\
1&2&1&0&0&0&0&0&0\\
1&0&0&1&0&0&0&0&0\\
1&1&0&1&1&0&0&0&0\\
1&2&1&1&2&1&0&0&0\\
1&0&0&2&0&0&1&0&0\\
1&1&0&2&2&0&1&1&0\\
1&2&1&2&4&2&1&2&1\\
\end{pmatrix}\,.
\end{equation}
Thus, the corresponding sampling formula for each $f\in V^2_\Phi$ reads:

\begin{cgris*}{}
{
\[
f(t,s)=\sum_{n,m\in \mathbb{Z}}\Big\{\sum_{k,k'=0}^2\,\Delta_+^{k,k'}f(a+3n,b+3m)\,T^{k,k'}_{a,b}(t-3n,s-3m)\Big\}\,, \quad t,s\in \mathbb{R}\,,
\]}
\end{cgris*}
where the sampling functions $T^{k,k'}_{a,b}(t,s)$ are obtained from the columns of the inverse matrix \eqref{i2} as:
\[
\begin{split}
T^{0,0}_{a,b}(t,s)&=S_{a,b}(t,s)+S_{a,b}(t,s-1)+S_{a,b}(t,s-2)+S_{a,b}(t-1,s)+S_{a,b}(t-1,s-1)
\\&+S_{a,b}(t-1,s-2)+S_{a,b}(t-2,s)+S_{a,b}(t-2,s-1)+S_{a,b}(t-2,s-2)\,; \\
T^{0,1}_{a,b}(t,s)&= S_{a,b}(t,s-1)+2S_{a,b}(t,s-2)+S_{a,b}(t-1,s-1)+2S_{a,b}(t-1,s-2)\\
&+S_{a,b}(t-2,s-1)+2S_{a,b}(t-2,s-2)\,;\\
T^{0,2}_{a,b}(t,s)&= S_{a,b}(t,s-2)+S_{a,b}(t-1,s-2)+S_{a,b}(t-2,s-2)\,;\\
T^{1,0}_{a,b}(t,s)&= S_{a,b}(t-1,s)+S_{a,b}(t-1,s-1)+S_{a,b}(t-1,s-2)+2S_{a,b}(t-2,s-2)\\
&+2S_{a,b}(t-2,s-1)+2S_{a,b}(t-2,s-2)\,;\\
T^{1,1}_{a,b}(t,s)&=S_{a,b}(t-1,s-1)+2S_{a,b}(t-1,s-2)+2S_{a,b}(t-2,s-1)+4S_{a,b}(t-2,s-2)\,; \\
T^{1,2}_{a,b}(t,s)&=S_{a,b}(t-1,s-2)+2S_{a,b}(t-2,s-2)\,;\\
T^{2,0}_{a,b}(t,s)&=S_{a,b}(t-2,s)+S_{a,b}(t-2,s-1)+S_{a,b}(t-2,s-2)\,; \\
T^{2,1}_{a,b}(t,s)&=S_{a,b}(t-2,s-1)+2S_{a,b}(t-2,s-2)\,;\\
T^{2,2}_{a,b}(t,s)&=S_{a,b}(t-2,s-2)\,, \quad t,s \in \mathbb{R}\,.
\end{split}
\]

Observe that $\mathcal{T}_{\Phi}\big(y_{3n+k,3m+k'}\big)=S_{a,b}(t-(3n+k),s-(3m+k'))$, where $y_{n,m}:={\rm e}^{-2\pi i n x}\,{\rm e}^{-2\pi i my}/\overline{K_{a,b}(x,y)}$, $n,m\in \mathbb{Z}$.

%%%%%%%%%%%%
\section{Conclusions}
%%%%%%%%%%%%
Starting from a Shannon-type sampling formula valid in a shift-invariant subspace, we have introduced a systematic way to derive stable sampling formulas involving differences and/or averages from the data available in the former sampling formula. The mathematical technique is very simple and relies on a representation of the elements of the shift-invariant subspace as an inner product in the auxiliary space $L^2(0,1)$; this allows us to identify an involved dual pair of Riesz bases by means of  an easy matrices calculation. Besides, a suitable isomorphism between $L^2(0,1)$ and the shift-invariant subspace gives the resulting sampling formulas. These formulas can be viewed as a multichannel filter processing.

Thus, a non-exhaustive variety of examples in the one-dimensional case are exhibited throughout the work. All these sampling formulas can be understood as a change of basis in the original Shannon-type sampling formula. The two-dimensional case is also treated: it reduces to the one-dimensional case via the Kronecker product of matrices. The same procedure could be generalized to higher dimensions.

Although the obtained formulas are Riesz bases expansions, as it was pointed out the used technique still remains valid for overcomplete expansions i.e., the frame setting. Furthermore, this technique also works in other types of sampling formulas:  for instance, those involving samples of filtered versions of the functions in the shift-invariant subspace.

\medskip

\noindent{\bf Acknowledgments:} 
This work has been supported by the grant MTM2014-54692-P from the Spanish {\em Ministerio de Econom\'{\i}a y Competitividad (MINECO)}
\vspace*{0.3cm}
%%%%%%%%%%%%%%%%%%%%%%%%%%%%%%%%%%%%%%%%%%%%%%%%%%%%%%%%%%%%%%%%%%%%%

%%%%%%%%%%%%%%%%%%%%%%%%%
\end{document}